\documentclass[
aps,%
12pt,%
final,%
notitlepage,%
oneside,%
onecolumn,%
nobibnotes,%
nofootinbib,%
superscriptaddress,%
noshowpacs,%
centertags]%
{revtex4}
\usepackage[cp1251]{inputenc}
\usepackage[russian]{babel}
\usepackage{amsmath, amsthm, amsfonts, amssymb, dsfont}
\usepackage{graphicx,color,srcltx}

\newtheorem{theorem}{Theorem}
\newtheorem{lemma}{Lemma}

\theoremstyle{remark}
\newtheorem{remark}{Remark}

\def\tht{\theta}
\def\Om{\Omega}
\def\om{\omega}
\def\e{\varepsilon}
\def\g{\gamma}
\def\G{\Gamma}
\def\l{\lambda}
\def\p{\partial}
\def\D{\Delta}

\def\E{\mbox{\rm e}}

\def\d{\delta}
\def\L{\Lambda}

\def\Odr{\mathcal{O}}
\def\H{W_2}

\def\iu{\mathrm{i}}

\def\He{\mathcal{H}_\e}
\def\he{\mathfrak{h}_\e}
\def\Hep{\mathcal{H}_\e^\tht}\def\hep{\mathfrak{h}_\e^\tht}

\def\Ho{\mathcal{H}_0}
\def\ho{\mathfrak{h}_0}

\def\Ups{\Upsilon}

 \DeclareMathOperator{\RE}{Re}
 \DeclareMathOperator{\spec}{\sigma}

\DeclareMathOperator{\essspec}{\sigma_{e}}

\allowdisplaybreaks

\begin{document}

\title{On band spectrum of Schr\"odinger operator in periodic system of domains coupled by small windows}

\author{\firstname{D.~I.}~\surname{Borisov}}
\email{BorisovDI@yandex.ru} \affiliation{Institute of Mathematics CS USC RAS \& Bashkir State Pedagogical University}

 
\begin{abstract}
We consider a periodic system of domains coupled by small windows. In such domain we study the band spectrum of a Schr\"odinger operator subject to Neumann condition.  We show that near each isolated eigenvalue of the similar operator but  in the periodicity cell, there are several non-intersecting bands of the spectrum for the perturbed operator. We also discuss the position of the points at which the band functions attain the edges of each band.
\end{abstract}

\maketitle

\section{Introduction}

In this paper we study the spectrum of a Schr\"odinger operator subject to Neumann condition in a periodic system of domains coupled by small windows, see Figure~1. The periodicity cell can be a bounded or unbounded domain. As windows close, the domain splits into a set of decoupled domains and the spectrum of the similar operator in such domain is an infinite sequence of isolated eigenvalues of infinite multiplicities. The infinite multiplicity appears since we consider the operator in an infinite sequence of identical decoupled domains. In the presence of windows, each such isolated eigenvalue generates a band in the spectrum of the perturbed operator. Our aim is to describe the structure of these bands.

There is a series of the works devoted to studying simlar problems \cite{BRT}, \cite{Na3}, \cite{Na5}, \cite{Pan}, \cite{Yo}. The usual assumption was that the periodicity cell is a bounded domain. The operator was either the Laplacian \cite{BRT}, \cite{Na3}, \cite{Pan}, \cite{Yo}, or a more general elliptic operator  \cite{Na5}. The boundary condition on the boundary were Dirichlet, Neumann or Steklov condition. The main result was either the estimates for the bands in the essential spectrum or the asymptotics for the band functions of the perturbed operator. In the latter case the instructive results were obtained under the assumption that the limiting eigenvalue is a simple one for the operator in a periodicity cell with no perturbation. Here the obtained asymptotics allowed the authors to describe the location of the band generated by the perturbed band functions. Once the limiting eigenvalue is multiple, in a general situation it splits into several perturbed band functions each of those can generate a band in the spectrum. At the same time, these bands can overlap and glue in this way into a bigger band. To the best of our knowledge, this situation was not analyzed before in details and no results are known for the structure of the band in the mentioned case.

\begin{figure}[t]
\includegraphics[scale=0.6]{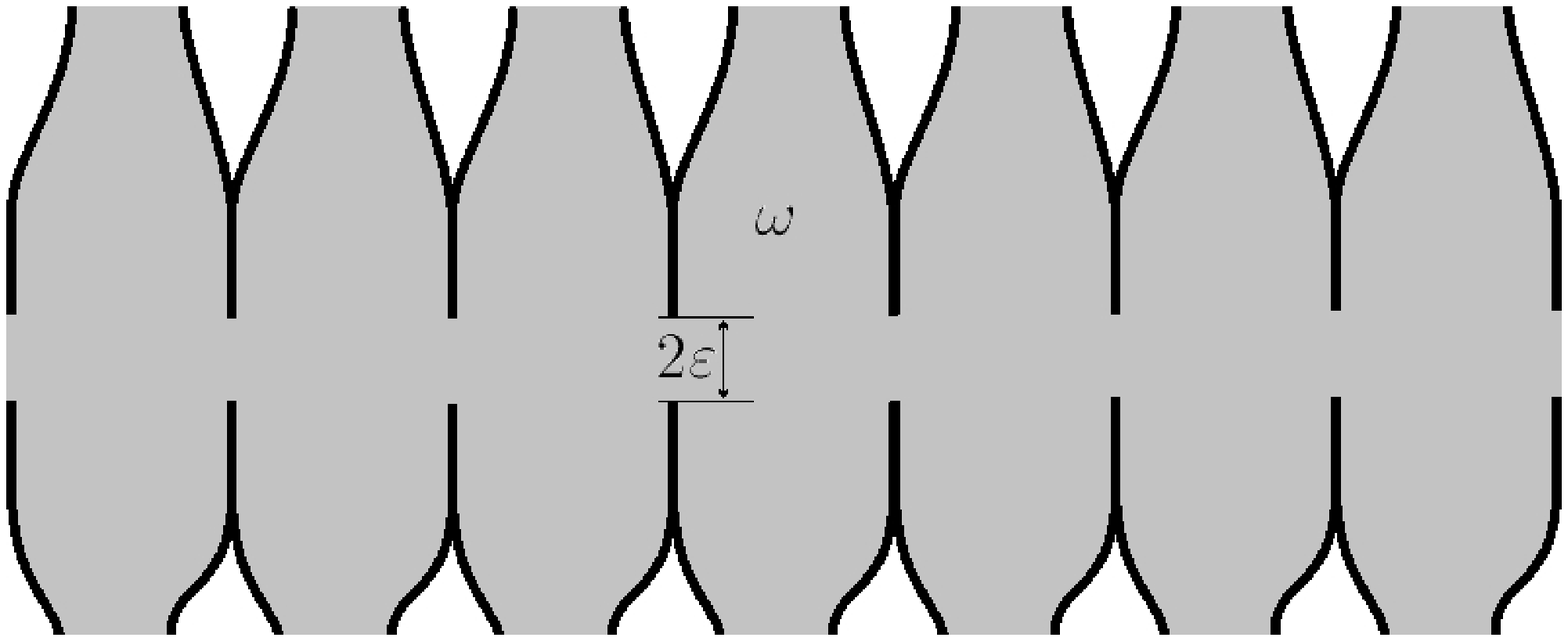}
\centerline{Figure 1. Domain}
\end{figure}

In the present paper  we consider the band functions converging to an isolated eigenvalue of the limiting operator in a periodicity cell. The multiplicity of the limiting eigenvalue can be arbitrary. Our result is the leading terms of the asymptotic expansions for the perturbed band functions converging to this limiting eigenvalue. It allows us to describe the structure of the associated bands. We show that in a general situation a double limiting eigenvalue generates two separated bands and their lengths are of different order. Moreover, the edges of a bigger band are attained by the associated band function in the center or at the end-points of the Brillouin zone. For the smaller band the situation can be different and the edges can be attained in the internal points of the Brillouin zone, see discussion in the next section. If the multiplicity of the limiting eigenvalue is three or more, there still exist two separated bands with the described properties. There can be also additional bands and we make a conjecture on their structure.

\section{Formulation of problem and main results}

Let $x=(x_1,x_2)$ be Cartesian coordinates in $\mathds{R}^2$ and $\Om$ be a domain in $\mathds{R}^2$ with piecewise $C^1$-boundary. We assume that domain  $\Om$ is invariant under the shift $(x_1,x_2)\to (x_1+1,x_2)$, the points 
$(n,0)$, $n\in \mathds{Z}$, are interior for $\Om$, and the lines 
$\{x: x_1=n\}$, $n\in \mathds{Z}$, are non-tangent to the boundary of $\Om$.  By $\e$ we denote a small positive parameter, and we let $\Om_\e:=\Om\setminus\bigcup\limits_{n\in\mathds{Z}}\{x: x_1=n, |x_2|\geqslant \e\}$.
The main object of our study is a Schr\"odinger operator $\He:=-\D+V$ in $L_2(\Om_\e)$ subject to the Neumann boundary condition on $\p\Om_\e$. 
Here $V\in L_\infty(\Om)$ is a real-valued potential satisfying the identity $V(x_1+1,x_2)=V(x_1,x_2)$ for all $x\in\Om$. Rigourously we define $\He$ as the self-adjoint operator in $L_2(\Om_\e)$ associated with the lower-semibounded closed symmetric sesquilinear form
\begin{equation*}
\he(u,v):=(\nabla u,\nabla v)_{L_2(\Om_\e)}+(Vu,v)_{L_2(\Om_\e)} \quad \text{in}\quad L_2(\Om_\e)
\end{equation*}
with the domain $\H^1(\Om_\e)$. Our main aim is to study the spectrum of $\He$ as $\e\to+0$.

Since operator $\He$ is periodic, we employ Floquet-Bloch theory to describe its spectrum. Namely, we introduce the periodicity cell $\om:=\Om\cap\{x: 0<x_1<1\}$. The cell can be either bounded or unbounded. For $\tht\in[0,2\pi)$ we consider the operator $\Hep:=-\D+V$ in $L_2(\Om)$ subject to the Neumann condition on $\p\om\setminus\g_\e$, $\g_\e:=\g_\e^+\cup\g_\e^-$, $\g_\e^+:=\{(1,x_2): |x_2|<\e\}$,  $\g_\e^-:=\{(0,x_2): |x_2|<\e\}$,  while on $\g_\e$ we impose quasi-periodic boundary conditions
\begin{equation}\label{2.1}
u\big|_{\g_\e^+}=\E^{\iu\tht} u\big|_{\g_\e^-},\quad \frac{\p u}{\p x_1}\bigg|_{\g_\e^+}=\E^{\iu\tht} \frac{\p u}{\p x_1}\bigg|_{\g_\e^-}.
\end{equation}
As operator $\He$, we introduce $\Hep$ as the self-adjoint operator associated with the sesquilinear form
\begin{equation*}
\hep(u,v):=(\nabla u,\nabla v)_{L_2(\Om)}+(Vu,v)_{L_2(\Om)}, \quad\text{in}\quad L_2(\om)
\end{equation*}
whose domain consists of the functions in $\H^1(\Om)$ satisfying first condition in (\ref{2.1}). The spectrum of $\He$ has a band structure and reads as
\begin{equation*}
\spec(\He)=\bigcup\limits_{\tht\in[0,2\pi)} \spec(\Hep),
\end{equation*}
where the symbol $\spec(\cdot)$ stands for the spectrum of an operator.

By $\Ho$ we denote the operator $-\D+V$ in $L_2(\om)$ subject to the Neumann condition on $\p\om$. As above, it is the self-adjoint operator associated with the sesquilinear form
\begin{equation*}
\ho(u,v):=(\nabla u,\nabla v)_{L_2(\om)}+(Vu,v)_{L_2(\om)}\quad \text{in}\quad L_2(\om)
\end{equation*}
on the domain $\H^1(\om)$. Since  $\om$ is not necessarily bounded, the spectrum of $\Ho$ can involve both essential and discrete part.

The main aim of this paper is to study the behavior of the spectrum of $\Hep$ as $\e\to+0$. Denoting by $\essspec(\cdot)$ the essential spectrum of an operator, by \cite{Bi} we know that $\essspec(\He)$ is independent of the boundary conditions on $\g_\e$ and thus $\essspec(\Hep)=\essspec(\Ho)$ for all $\e$ and $\tht$. Hence, it is a fixed set independent of $\e$ and $\tht$. At the same time, the eigenvalues of $\Hep$ do depend both on $\e$ and $\tht$. By analogy with \cite{Ga} one can show that the discrete eigenvalues of $\Hep$ converge either to the discrete eigenvalues of $\Ho$ or to the thresholds of $\essspec(\Ho)$. In the former case the total multiplicity is preserved. The convergence of the eigenvalues is uniform in $\tht$.

Let $\l_0$ be a discrete eigenvalue of $\Ho$ of multiplicity $k$ and $\psi_0^{(j)}$, $j=1,\ldots,k$, are the associated eigenfunctions orthonormalized in $L_2(\om)$. For any linear combination $\psi$ of $\psi_j$ we define
\begin{equation*}
l_\tht(\psi):=\psi(M_+)\E^{-\iu\tht}-\psi(M_-),\quad l_\tht'(\psi):=\frac{\p\psi}{\p x_2}(M_+)\E^{-\iu\tht}+\frac{\p\psi}{\p x_2}(M_-).
\end{equation*}
We denote
$\Psi_0^{(j)}(x,\tht):=\sum\limits_{i=1}^{n} a_{ji}(\tht)\psi_0^{(i)}(x)$,
where $a=(a_{ij}(\tht))_{i,j=1,\ldots,n}$ is a unitary matrix with complex entries. We fix this matrix by the requirements
\begin{align}
l_\tht(\Psi_0^{(j)})=0,
\quad
j=2,\ldots,k,\qquad l_\tht'(\Psi_0^{(j)})=0,\quad j=3,\ldots, k.\label{3.5a}
\end{align}
where $M_-:=(0,0)$, $M_+:=(1,0)$. We shall show later in Lemma~\ref{lm3.1} that these identities can be satisfied.

We indicate $L:=(l_\tht(\psi_1),\ldots,l_\tht(\psi_k))$, $L':=(l_\tht'(\psi_1),\ldots,l_\tht'(\psi_k))$.

Our  main result is

\begin{theorem}\label{th1}
Suppose
\begin{equation}\label{2.10}
|\psi_1(M_+)|\not=|\psi_1(M_-)|.
\end{equation}
There exist exactly $k$ eigenvalues $\l_\e^{(j)}$, $j=1,\ldots,k$, of $\Hep$ (counting multiplicity) converging to $\l_0$ as $\e\to+0$. They can be ordered so that first two of them have the asymptotics
\begin{align}
&\l_\e^{(1)}(\tht)=\l_0+\l_{0,1}^{(1)}(\tht)\ln^{-1}\e+\Odr(\ln^{-2}\e),\label{2.5}
\\
&\l_\e^{(2)}(\tht)=\l_0+\l_{0,2}^{(2)}(\tht)\e^2+\Odr(\e^2\ln^{-1}\e),\label{2.6}
\end{align}
while the other satisfy
\begin{equation}\label{2.7}
\l_\e^{(j)}(\tht)=\l_0+\Odr(\e^3),\quad j=3,\ldots,k.
\end{equation}
Here
\begin{equation}
\l_{0,1}^{(1)}(\tht)=-\frac{\pi}{2} \|L\|_{\mathds{C}^k}^2,
\quad \l_{1,0}^{(2)}=\frac{\pi}{8}
\frac{\|L\|_{\mathds{C}^k}^2\|L'\|_{\mathds{C}^k}^2
-|(L',L)_{\mathds{C}^k}|^2}{\|L\|_{\mathds{C}^k}^2}.
\label{2.8a}
\end{equation}
The estimates for the error terms are uniform in $\tht\in[0,2\pi)$.
\end{theorem}

This theorem gives an opportunity to describe the position of bands in the spectrum for $\He$. Namely, let $\l_0$ be a discrete eigenvalue of $\Ho$. Then in a small but fixed neighborhood of $\l_0$ the spectrum of $\He$ is generated by the eigenvalues $\l_\e^{(j)}(\tht)$, $j=1,\ldots,k$, described in Theorem~1, and these bands are $\{\l_\e^{(j)}(\tht): \tht\in[0,2\pi)\}$. As it follows from (\ref{2.6}), the  first of these bands is the segment
\begin{equation*}
\Ups_\e^{(1)}:=[\l_0+\L_1^-|\ln\e|^{-1}+\Odr(|\ln\e|^{-2}), \l_0+\L_1^+|\ln\e|^{-1}+\Odr(|\ln\e|^{-2})],
\end{equation*}
where $\L_1^-:=-\max\limits_{[0,2\pi]} \l_{0,1}^{(1)}(\tht)$,
$\L_1^+:=-\min\limits_{[0,2\pi]} \l_{0,1}^{(1)}(\tht)$. It is easy to see that the function $\tht\mapsto \l_{0,1}^{(1)}(\tht)$ achieves their minima and maxima at the points $\tht=0$, $\tht=\pi$, $\tht=2\pi$, and due to (\ref{2.10}) both $\L_1^-$ and $\L_1^+$ are strictly positive. If $\l_0$ is a degenerate eigenvalue, then there exists one more band
\begin{equation*}
\Ups_\e^{(2)}:=[\l_0+\L_2^-\e^{2}+\Odr(\e^2|\ln\e|^{-1}), \l_0+\L_2^+\e^2+\Odr(\e^2|\ln\e|^{-1})],
\end{equation*}
where $\L_2^-:=\min\limits_{[0,2\pi]} \l_{1,0}^{(2)}(\tht)$,
$\L_2^+:=\max\limits_{[0,2\pi]} \l_{1,0}^{(2)}(\tht)$. Comparing this band with $\Ups_\e^{(1)}$, we see that these bands do not intersect and they are located to the right from $\l_0$. There  is a gap between these bands; its length is of order $\Odr(|\ln\e|^{-1})$. Another really interesting feature of the second band is the values of $\tht$ at which the function $\l_\e^{(2)}(\tht)$ attains the edges of $\Ups_\e^{(2)}$. These values of $\tht$ do depend of the numbers $\psi_j(M_+)$, $\psi_j(M_-)$, $\frac{\p\psi_j}{\p x_2}(M_+)$, $\frac{\p\psi_j}{\p x_2}(M_-)$ and they can be found as the zeroes of $\big(\l_{0,2}^{(2)}\big)'$ up to a small error. We do not provide these calculations here since they are very bulky. Instead of this we just adduce two typical graphs of $\l_{0,2}^{(2)}$ which we made for $k=2$ assuming various values of $\psi_j(M_\pm)$, $\frac{\p\psi_j}{\p x_2}(M_\pm)$, see Figures~1,~2. In Figure~1 we assume that
\begin{align*}
&\psi_1(M_+)=1,\quad \psi_1(M_-)=2,\quad\psi_2(M_+)=1,\quad \psi_1(M_-)=3,\quad
\\
&\frac{\p\psi_1}{\p x_2}(M_+)=\frac{3}{2},\quad \frac{\p\psi_1}{\p x_2}(M_-)=\frac{5}{2},\quad \frac{\p\psi_2}{\p x_2}(M_+)=\frac{1}{2},\quad \frac{\p\psi_2}{\p x_2}(M_-)=2.
\end{align*}
while in Figure~2 similar identities are
\begin{align*}
&\psi_1(M_+)=1,\quad \psi_1(M_-)=2,\quad\psi_2(M_+)=-1,\quad \psi_1(M_-)=3,\quad
\\
&\frac{\p\psi_1}{\p x_2}(M_+)=\frac{3}{2},\quad \frac{\p\psi_1}{\p x_2}(M_-)=\frac{5}{2},\quad \frac{\p\psi_2}{\p x_2}(M_+)=-\frac{1}{2},\quad \frac{\p\psi_2}{\p x_2}(M_-)=2.
\end{align*}
As we see, in the former case the minima and maxima are attained at $\tht=0$, $\tht=\pi$, $\tht=2\pi$, while in the latter case the maxima are attained at two internal points located symmetrically w.r.t. $\tht=\pi$. This can be regarded as one more example of a periodic operator with a band function attaining its extrema in internal points of the Brillouin zone, see recent discussion in \cite{BP1}, \cite{BP2}, \cite{BP3}, \cite{Na2}. Of course, there is still an open question on existence of a domain $\om$ and a potential $V$ giving required values of $\psi_j(M_\pm)$, $\frac{\p\psi_j}{\p x_2}(M_\pm)$, but we believe that this is a rather general situation for non-symmetric domains and potentials.

In the case of a degenerate eigenvalue $\l_0$ of multiplicity greater than two, in addition to the above described bands there are also additional bands associated with $\l_\e^{(j)}(\tht)$, $j\geqslant 3$. Our formulae (\ref{2.7}) do not give any precise information about the bands location. They just say that the distance from these bands to $\l_0$ is of order at most $\Odr(\e^3)$ and hence in a general situation they are separated from two aforementioned bands. Nevertheless, proceeding as in the proof of Theorem~\ref{th1}, one can construct the leading terms in the asymptotics for all the eigenvalues $\l_\e^{(j)}(\tht)$, $j\geqslant 3$. We expect that these terms are of different orders like for $\l_\e^{(1)}$ and $\l_\e^{(2)}$, i.e., for greater index $j$ the leading term for $\l_\e^{(j)}$ is smaller. We then conjecture that in a general situation there are $k$ separated bands associated $\l_\e^{(j)}(\tht)$, $j=1,\ldots,k$, and the lengths of these bands are of different order of smallness. In Figure~4 we show schematically the expected structure of these bands. It is also natural to expect that the edges of these bands can be attained by the associated band functions in the internal points of the Brillouin zone.

\section{Proof of Theorem~\ref{th1}}

To construct the asymptotics for $\l_\e^{(j)}$, $j=1,\ldots,k$, we follow a standard scheme. Namely, first we construct them formally by the method of matching asymptotic expansions \cite{Il}, and after that we need to justify the formal asymptotics by estimating the error terms. The latter is very standard and can be done completely in the same way as, for instance, in \cite{Izv03}, \cite{Ga1}. This is why in this section we address only the formal construction.

We begin with an auxiliary lemma.

\begin{tabular}{ccc}
\includegraphics[scale=0.3]{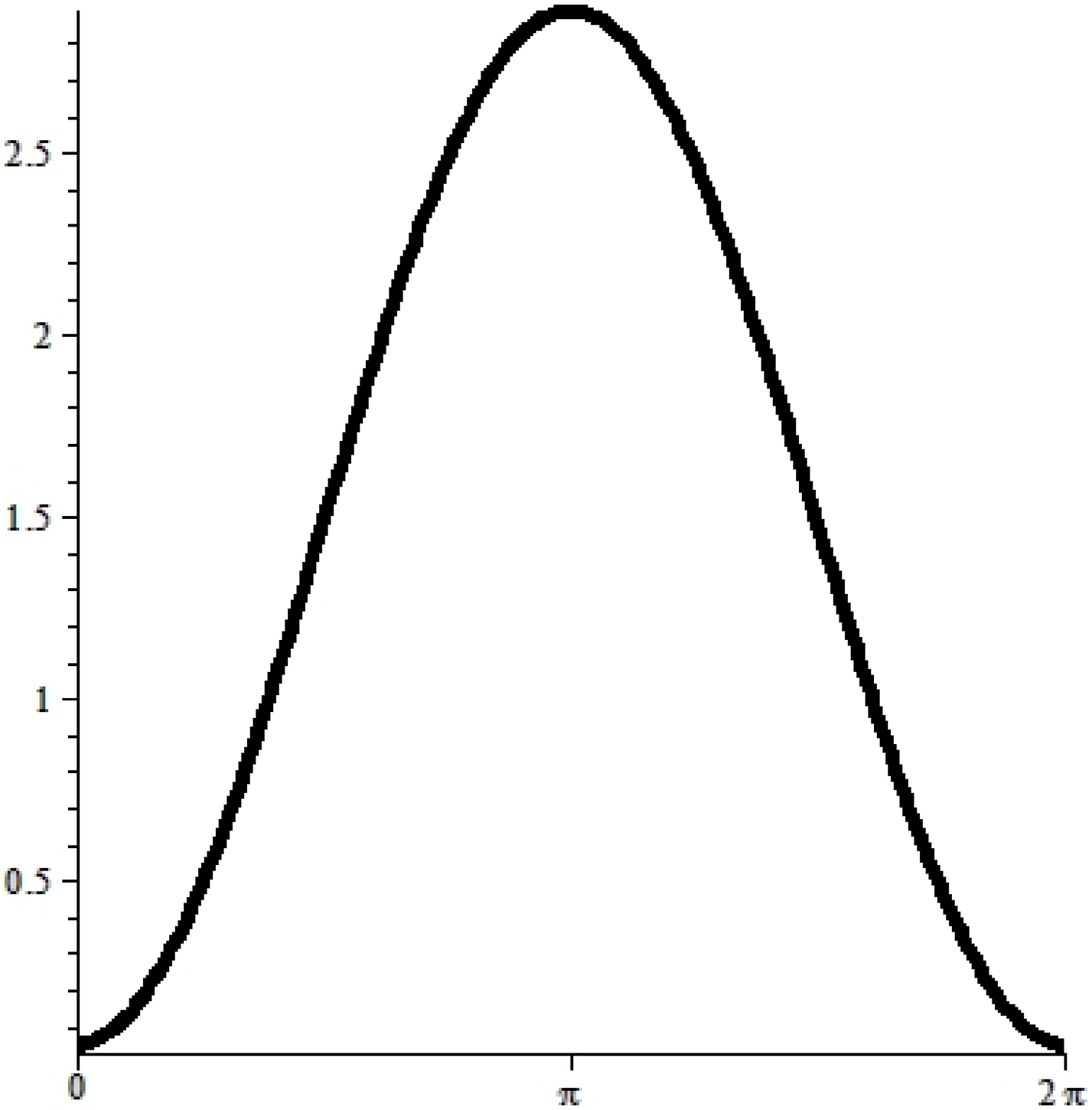}
&\hspace{2 true cm} &
\includegraphics[scale=0.3]{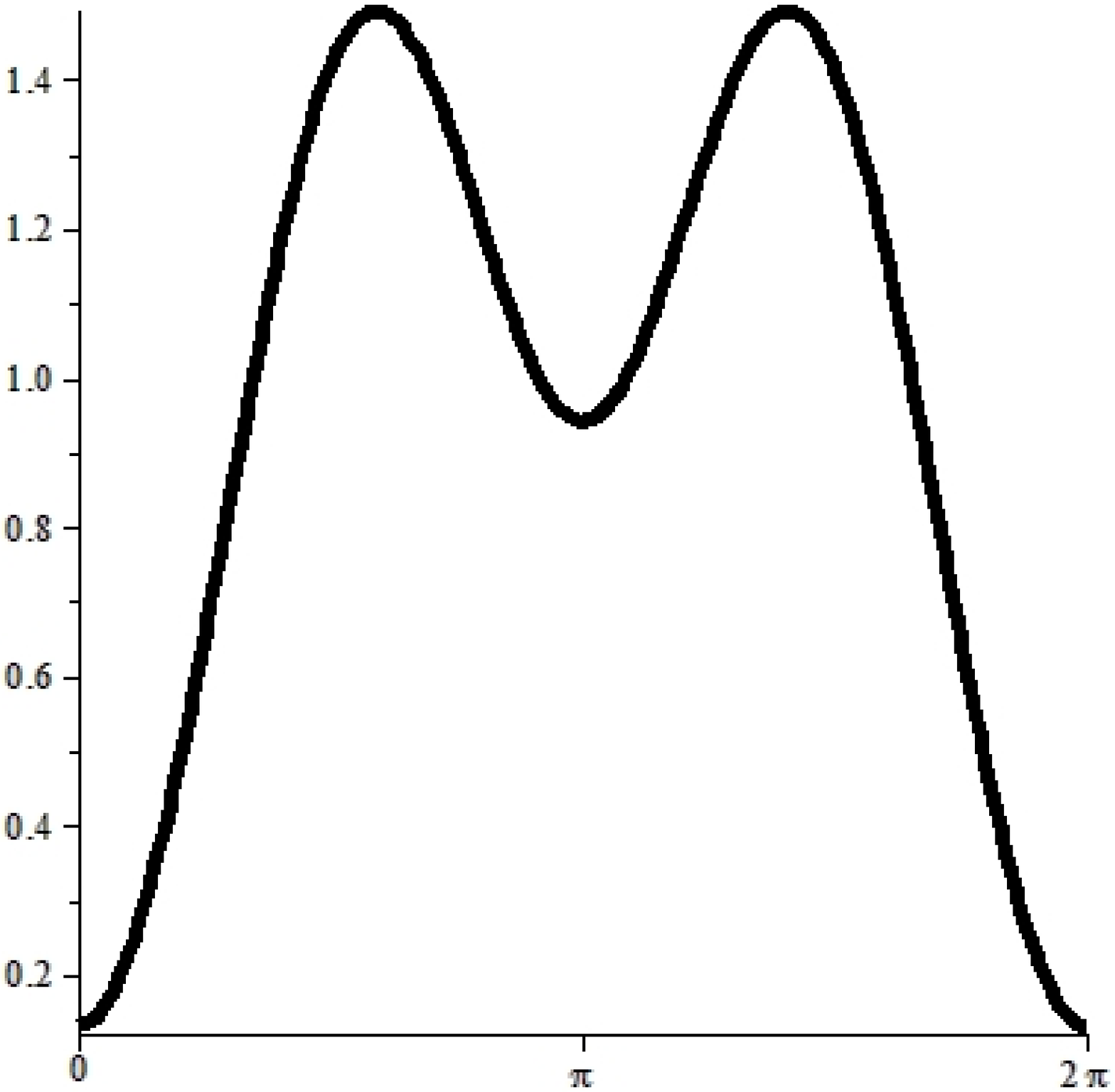}
\\
\parbox{5 true cm}{\small Figure 2: Maxima in internal points  of Brillouin zone} &  & \parbox{5 true cm}{\small Figure 3: Maximum in the center  of Brillouin zone}
\end{tabular}
\bigskip
\bigskip

\begin{lemma}\label{lm3.1}
Suppose (\ref{2.10}). Functions $a_{ij}(\tht)$ can be chosen so that relations 
(\ref{3.5a}) 
are satisfied. At that,
\begin{equation}\label{3.22}
|l_\tht(\Psi_0^{(1)})|^2= \|L\|_{\mathds{C}^k}^2,
\quad
|l'_\tht(\Psi_0^{(2)})|^2=\frac{\|L\|_{\mathds{C}^k}^2\|L'\|_{\mathds{C}^k}^2
-|(L',L)_{\mathds{C}^k}|^2}{\|L\|_{\mathds{C}^k}^2}.
\end{equation}
\end{lemma}

We shall prove this lemma later after the formal construction of the asymptotics. Throughout the formal construction the lemma is assumed to hold true.

By $\psi_\e^{(j)}$ we denote indicate the eigenfunctions associated with $\l_\e^{(j)}$. Outside small neighborhoods of points $M_\pm$ we construct the asymptotics for $\psi_\e^{(j)}$ as
\begin{align}
&\psi_\e^{(1)}(x,\tht)=\Psi_0^{(1)}(x,\tht)+\ln^{-1}\e \Psi_{0,1}^{(1)}(x,\tht)+\Odr(\ln^{-2}\e),\label{3.1}
\\
&\psi_\e^{(2)}(x,\tht)=\Psi_0^{(2)}(x,\tht)+\e^2 \Psi_{1,0}^{(2)}(x,\tht)+\Odr(\e^2\ln^{-1}\e),\label{3.2}
\\
&\psi_\e^{(j)}(x,\tht)=\Psi_0^{(j)}(x,\tht)+\Odr(\e^3),\quad j=3,\ldots,k.\label{3.3}
\end{align}
The asymptotics for the perturbed eigenvalues are constructed as (\ref{2.5}), (\ref{2.6}), (\ref{2.7}), where the coefficients $\l_{0,1}^{(1)}$, $\l_{1,0}^{(2)}$ are to be determined. We substitute these asymptotics into the eigenvalue equation for $\Hep$ (rewritten as a boundary value problem) and equate the coefficients at the like powers of $\e$ and $\ln\e$. It gives the boundary value problems for $\Psi_{0,1}^{(1)}$ and $\Psi_{1,0}^{(2)}$:
\begin{equation}\label{3.11}
(-\D+V-\l_0)\Psi_{p,q}^{(j)}=\l_{p,q}^{(j)}\Psi_0^{(j)}\quad \text{in}\quad \om,\qquad \frac{\p\Psi_{p,q}^{(j)}}{\p\nu}=0\quad \text{on} \quad \p\om,
\end{equation}
where $(j,p,q)=(1,0,1)$ or $(j,p,q)=(2,1,0)$. As in \cite{Na1}, \cite{Ga1}, we shall need solutions to these problems with certain singularities at points $M_\pm$.
We shall determine the structure of these singularities later as a result of matching exterior expansions (\ref{3.1}), (\ref{3.2}), (\ref{3.3}) and inner ones. The latter are constructed in terms of rescaled variables $\xi^\pm:=(\xi^\pm_1,\xi^\pm_2)$,  $\xi_1^-:=-x_2\e^{-1}$, $\xi_2^-:=x_1\e^{-1}$, $\xi_1^+:=x_2\e^{-1}$, $\xi_2^+:=(1-x_1)\e^{-1}$ as follows,
\begin{align}
&\psi_\e^{(1)}(x,\tht)=\ln^{-1}\e \Phi_{0,\pm}^{(1)}(\xi^\pm,\tht,\e)+\Odr(\ln^{-2}\e),
\nonumber
\\
&\psi_\e^{(2)}(x,\tht)=\e \Phi_{1,\pm}^{(2)}(\xi^\pm,\tht,\e)+\Odr(\e\ln^{-1}\e),\label{3.13}
\\
&\psi_\e^{(j)}(x,\tht)=\Odr(\e^2),\quad j=3,\ldots,k.\label{3.14}
\end{align}
These expansions are used in vicinities of points $M_\pm$, respectively. We substitute these expansions and (\ref{2.5}), (\ref{2.6}), (\ref{2.7}) into the eigenvalue equation for $\Hep$ and rewrite it as a boundary value problem, pass to $\xi^\pm$ and equate the coefficients at the like powers of $\e$ and $\ln^{-1}\e$. It yields the boundary value problem for $\Phi_{p,q}^{(j)}$:
\begin{equation}\label{3.15}
\begin{aligned}
-&\D\Phi_{p,\pm}^{(j)}=0\quad \text{as} \quad \xi_2>0,\qquad \frac{\p\Phi_{p,\pm}^{(j)}}{\p\xi_2}=0\quad \text{on}\quad \G,\qquad \frac{\p\Phi_{p,+}^{(j)}}{\p \xi_2}=-\E^{\iu\tht}\frac{\p\Phi_{p,-}^{(j)}}{\p \xi_2}\quad \text{on}\quad \g,
\\
&\Phi_{p,+}^{(j)}=\E^{\iu\tht}\Phi_{p,-}^{(j)}\quad \text{on}\quad \g,\quad \xi=(\xi_1,\xi_2),\quad \g:=\{\xi: |\xi_1|<1, \xi_2=0\},\quad \G:=O\xi_1\setminus\overline{\g}.
\end{aligned}
\end{equation}

\begin{figure}[t]
\includegraphics[scale=0.5]{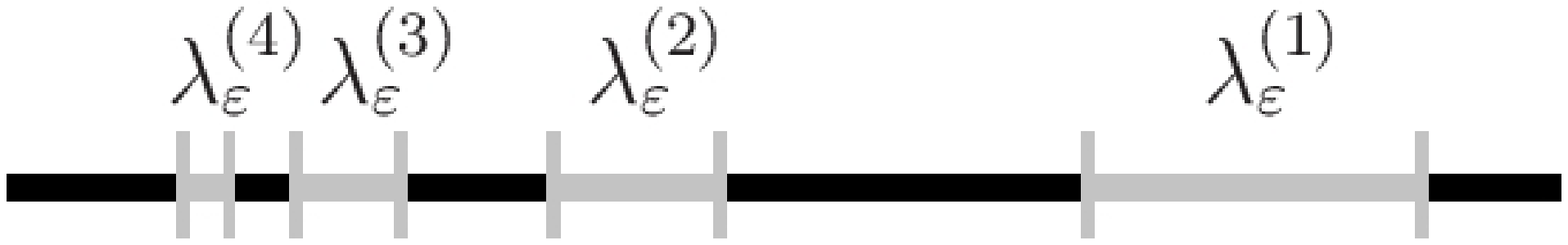}
\centerline{Figure 4. Band spectrum}
\end{figure}

We let $X_0(\xi):=\RE\ln(\sin z+\sqrt{\sin^2 z-1})$, $z=\xi_1+\iu\xi_2$, where the branches of the logarithm and the square root are fixed by the requirements $\ln 1=0$, $\sqrt{1}=1$. This function is a harmonic one as $\xi_2>0$ and satisfies Dirichlet condition on $\g$ and Neumann one on $\G$.
We choose functions $\Phi_{0,\pm}^{(1)}$ as
\begin{equation}\label{3.18}
\Phi_{0,\pm}^{(1)}(\xi,\tht,\e)=b_{0,\pm}^{(1,1)} X_0(\xi)+b_{0,\pm}^{(1,2)}\ln\e,
\end{equation}
write its asymptotics as $\xi^\pm\to\infty$ and pass to variables $x$. It implies
\begin{align*}
\Phi_{0,\pm}^{(1)}(\xi_\pm,\tht,\e)=\big(b_{0,\pm}^{(1,2)}(\tht)- b_{0,\pm}^{(1,1)}(\tht)\big)\ln\e+ b_{0,\pm}^{(1,1)}\big(\ln|x-M_\pm|+\ln 2\big)+\Odr(\e^2|x-M_\pm|^{-2}).
\end{align*}
According to the method of matching asymptotic expansions, this asymptotics should coincide with the asymptotic behavior of the inner expansion as $x\to M_\pm$. Expanding $\Psi_0^{(1)}$ by Taylor formula as $x\to M_\pm$ and writing the matching condition, we obtain that $\Psi_{0,1}^{(1)}$ should have the singularities
\begin{equation}\label{3.16}
\Psi_{0,1}^{(1)}(x,\tht)=b_{0,\pm}^{(1,1)}\ln|x-M_\pm|+\Odr(1),\quad x\to M_\pm,
\end{equation}
and the identities
\begin{equation}\label{3.17}
b_{0,\pm}^{(1,2)}(\tht)-b_{0,\pm}^{(1,1)}(\tht)=\Psi_0^{(1)}(M_\pm,\tht)
\end{equation}
should be satisfied. As it follows from (\ref{3.18}), the boundary conditions on $\g$ in (\ref{3.15}) hold true if and only if $b_{0,+}^{(1,2)}=\E^{\iu\tht}b_{0,-}^{(1,2)}$,
$b_{0,+}^{(1,1)}=-\E^{\iu\tht}b_{0,-}^{(1,1)}$, and by (\ref{3.17}) it yields
\begin{align*}
b_{0,\pm}^{(1,1)}=\frac{1}{2}\big(\Psi_0^{(1)}(M_\mp,\tht)\E^{\pm\iu\tht} - \Psi_0^{(1)}(M_\pm,\tht)\big),\quad b_{0,\pm}^{(1,2)}=\frac{1}{2} \big( \Psi_0^{(1)}(M_\mp,\tht)\E^{\pm\iu\tht} + \Psi_0^{(1)}(M_\pm,\tht)\big).
\end{align*}
The solvability conditions of boundary value problem (\ref{3.11}), (\ref{3.16}) is obtained in a standard way by integrating by parts with the singularities taken into account (see, for instance, \cite[Ch. I\!I\!I, Sec. 3, Lm. 3.1]{Il}):
\begin{equation}
\l_{0,1}^{(1)}(\tht)\d_{1j}=-\frac{\pi}{2} \overline{\big(
\Psi_0^{(j)}(M_+,\tht)\E^{-\iu\tht}-\Psi_0^{(j)}(M_-,\tht)
\big)}\big(
\Psi_0^{(1)}(M_+,\tht)\E^{-\iu\tht}-\Psi_0^{(1)}(M_-,\tht)
\big),
\end{equation}
as $j=1,\ldots,k$. Due to (\ref{3.5a}), these identities are satisfied for $j=2,\ldots,k$, while as $j=1$, by Lemma~\ref{lm3.1} it implies the first formula in (\ref{2.8a}).

\begin{remark}\label{rem1}
If we assume the presence of the terms of order $\Odr(\ln^{-1}\e)$ in (\ref{3.2}), (\ref{3.3}), (\ref{3.13}), (\ref{3.14}), it can be constructed in the same way as $\Psi_{0,1}^{(1)}$ and $\Phi_{0,\pm}^{(1)}$. But then due to (\ref{3.5a}) such terms vanish and this is why we apriori do not write them in the asymptotics.
\end{remark}

The construction of functions $\Psi_{1,0}^{(2)}$ and $\Phi_{1,\pm}^{(2)}$ is similar to the above arguments. Functions $\Phi_{1,\pm}^{(2)}$ read as
\begin{equation*}
\Phi_{1,\pm}^{(2)}(\xi,\tht,\e)=b_{1,\pm}^{(2,1)}X_1(\xi)+b_{1,\pm}^{(2,2)}\xi_1,\quad X_1(\xi):=\RE\sqrt{z^2-1},
\end{equation*}
while the singularities of $\Psi_{1,0}^{(2)}$ are
\begin{equation}\label{3.20}
\Psi_{1,0}^{(2)}=-\frac{b_{1,\pm}^{(2,1)}}{2}r_\pm^{-1}\cos\phi_\pm + b_{1,\pm}^{(2,3)} \ln|x-M_\pm|+\Odr(1),\quad x\to M_\pm.
\end{equation}
Here $(r_-,\phi_-)$ are the polar coordinates associated with $(-x_2,x_1)$, while $(r_+,\phi_+)$ are the polar coordinates associated with $(x_2,1-x_1)$. The identities $b_{1,+}^{(2,2)}=\E^{\iu\tht}b_{1,-}^{(2,2)}$,
$b_{1,+}^{(2,1)}=-\E^{\iu\tht}b_{1,-}^{(2,1)}$,
\begin{equation*}
b_{1,-}^{(2,1)}+b_{1,-}^{(2,2)}=-\frac{\p\Psi_0^{(2)}}{\p x_2}(M_-,\tht),\quad b_{1,+}^{(2,1)}+b_{1,+}^{(2,2)}=\frac{\p\Psi_0^{(2)}}{\p x_2}(M_+,\tht)
\end{equation*}
should be obeyed. One more identity to be satisfied is $b_{1,+}^{(2,3)}=-\E^{\iu\tht}b_{1,-}^{(2,3)}$; it appears as a result of matching subsequent terms in the asymptotics. These identity allow us to determine $b_{q,\pm}^{(2,p)}$:
\begin{equation}\label{3.23}
\begin{aligned}
&b_{1,\pm}^{(2,1)}=\pm\frac{1}{2}\left(\E^{\pm\iu\tht}\frac{\p\Psi_0^{(2)}}{\p x_2}(M_\mp,\tht)+\frac{\p\Psi_0^{(2)}}{\p x_2}(M_\pm,\tht)\right),
\\
&b_{1,\pm}^{(2,2)}=\pm\frac{1}{2}\left(\frac{\p\Psi_0^{(2)}}{\p x_2}(M_\pm,\tht)-\E^{\pm\iu\tht}\frac{\p\Psi_0^{(2)}}{\p x_2}(M_\mp,\tht)\right).
\end{aligned}
\end{equation}
In view of the above identities
the solvability conditions of problem (\ref{3.15}), (\ref{3.20}) for $\Psi_{1,0}^{(2)}$ are
\begin{equation*}
\l_{1,0}^{(2)}\d_{2j}=
\pi b_{1,-}^{(2,3)}\overline{\big(\Psi_0^{(j)}(M_-,\tht) -\E^{-\iu\tht}\Psi_0^{(j)}(M_+,\tht)\big)}-\frac{\pi b_{1,-}^{(2,1)}}{4} \overline{\Big(\frac{\p\Psi_0^{(j)}}{\p x_2}(M_-,\tht)+\E^{-\iu\tht}\frac{\p\Psi_0^{(j)}}{\p x_2}(M_+,\tht)\Big)}.
\end{equation*}
As $j=3,\ldots,k$, these identities hold true due to (\ref{3.5a}). As $j=2$, we get the latter formula in (\ref{2.8a}), and by the identity for $j=1$ we can determine $b_{1,-}^{(2,3)}$.

\begin{remark} If we try to construct the terms of order $\Odr(\e^3)$ in (\ref{3.3}) and of order $\Odr(\e^2)$ in (\ref{3.14}), they vanish due to (\ref{3.5a}).
\end{remark}

The formal construction of leading terms in the asymptotics is complete.

\begin{proof}[Proof of Lemma~\ref{lm3.1}]
We let $a_{1j}(\tht):=\overline{l_\tht(\psi_j)} \|L\|_{\mathds{C}^k}^{-1}$,
$j=1,\ldots,k$. These quantities are well-defined for all $\tht\in[0,2\pi)$ since $l_\tht(\psi_1)\not=0$ for all $\tht$ thanks to (\ref{2.10}). We also assume that the vectors $(a_{i1},\ldots,a_{ik})$, $i=1,\ldots,k$ are orthonormalized in $\mathds{C}^k$. Then it is easy to check that functions $\Psi_0^{(j)}$ are orthonormalized in $L_2(\om)$  and the former identities in (\ref{3.5a}) and (\ref{3.23}) are satisfied.

In order to obey the latter identities in (\ref{3.5a}) and (\ref{3.23}), we first consider the functions $\widetilde{\psi}_i(x,\tht):=l_\tht(\psi_i)\psi_1(x)-l_\tht(\psi_1)\psi_i(x)$, $i=2,\ldots,k$. Due to (\ref{3.22}), we have $l_\tht(\psi_1)\not=0$ for all $\tht\in[0,2\pi)$ and this is why functions $\widetilde{\psi}_i$ are linear independent. It is straightforward to check that these functions satisfy $l_\tht(\widetilde{\psi}_i)=0$, $(\widetilde{\psi}_i,\Psi_0^{(1)})_{L_2(\om)}=0$, $i=2,\ldots,k$. Hence, the needed functions $\Psi_0^{(j)}$, $j=2,\ldots,k$, are linear combinations of $\widetilde{\psi}_i$, $i=2,\ldots,k$. The Gram matrix $G=\big((\widetilde{\psi}_i,\widetilde{\psi}_j)_{L_2(\om)} \big)_{i,j=2,\ldots,k}$  reads as $G=\widetilde{L}^*\widetilde{L}+|l_\tht(\psi_1)|^2E$, $\widetilde{L}:=(l_\tht(\psi_2),\ldots,l_\tht(\psi_k))$, $E$ is the unit $(k-1)\times(k-1)$ matrix, $^*$ denotes the adjoint matrix.  Matrix $G$ is positive definite and Hermitian. Hence, the matrix $G^{-1/2}$ is well-defined.

We let
\begin{equation*}
z_2:=\big|G^{-1/2}(l'_\tht(\psi_1)\widetilde{L}-l_\tht(\psi_1)\widetilde{L}')\big|^{-1} G^{-1/2}(l'_\tht(\psi_1)\widetilde{L}-l_\tht(\psi_1)\widetilde{L}'),\quad \widetilde{L}:=(l'_\tht(\psi_2),\ldots,l'_\tht(\psi_k)),
\end{equation*}
and by $z_3$, \ldots, $z_k$ we denote vectors in $\mathds{C}^{k-1}$ so that $z_i$, $i=2,\ldots,k$, form an orthonormalized basis in $\mathds{C}^k$. If $G^{-1/2}(l'_\tht(\psi_1)\widetilde{L}-l_\tht(\psi_1)\widetilde{L}'=0$, as $z_2$, \ldots, $z_k$ we take the standard basis in $\mathds{C}^{k-1}$. Then it is easy to check  that the functions $\Psi_0^{(i)}=\sum\limits_{j=2}^{k}b_{ij}\widetilde{\psi}_j$, $i=2,\ldots,k$, $(b_{ij})_{ij=2,\ldots,k}=(z_2 \cdots z_k)^* G^{-1/2}$ satisfy the latter identities in (\ref{3.5a}) and
\begin{equation*}
|l'_\tht(\Psi_0^{(2)})|^2=\big|G^{-1/2}(l'_\tht(\psi_1)\widetilde{L} -l_\tht(\psi_1)\widetilde{L}')\big|^2
=\big(l'_\tht(\psi_1)\widetilde{L}-l_\tht(\psi_1)\widetilde{L}', G^{-1}(l'_\tht(\psi_1)\widetilde{L}-l_\tht(\psi_1)\widetilde{L}')
\big)_{\mathds{C}^{k-1}}.
\end{equation*}
Employing this identity and
\begin{equation*}
G^{-1}=\frac{1}{|l_\tht(\psi_1)|^2}\left(E-\frac{\widetilde{L}^*\widetilde{L}} {\|L\|_{\mathds{C}^{k-1}}^2}\right),
\end{equation*}
by straightforward calculations we arrive at the latter identity in (\ref{3.22}).
\end{proof}

\begin{acknowledgments}

The author thanks Konstantin Pankrashkin for attracting my attention to this problem and for valuable discussions.

The work is partially supported by RFBR, by the grant of the President of Russia for young scientists-doctors of sciences (MD-183.2014.1) and the fellowship of Dynasty foundation for young mathematicians.

\end{acknowledgments}

\end{document}